\documentclass{article}
\usepackage{graphicx} 
\usepackage{hyperref}
\usepackage{subcaption}
\usepackage{amsmath, amssymb, amsthm}
\usepackage{xcolor}
\usepackage[toc,page]{appendix}
\usepackage{enumitem}

\DeclareMathOperator{\LO}{LO}
\DeclareMathOperator{\SG}{SG}
\newcommand{\sLO}{\mathcal{LO}}
\newcommand{\card}[1]{\lvert #1 \rvert}
\let\union\cup
\let\bigunion\bigcup
\let\partitions\vdash
\newcommand{\vocab}[1]{\emph{#1}}
\newcommand{\ceiling}[1]{\lceil #1 \rceil}

\newcommand{\powerset}{\mathcal{P}}
\newcommand{\Tset}{\mathcal{T}}

\title{Counting Strict Gridlock on Graphs}
\author{Matthew I. Jones, Zachary Winkeler}
\date{}

\usepackage{hyperref}
\hypersetup{
    colorlinks,
    linkcolor={red!50!black},
    citecolor={blue!50!black},
    urlcolor={blue!80!black}
}
\usepackage[capitalise,noabbrev,nameinlink]{cleveref}

\theoremstyle{plain}
\newtheorem{theorem}{Theorem}
\newtheorem{lemma}[theorem]{Lemma}
\newtheorem{proposition}[theorem]{Proposition}
\newtheorem{corollary}[theorem]{Corollary}

\theoremstyle{definition}
\newtheorem{example}[theorem]{Example}
\crefname{example}{Example}{Examples}

\begin{document}

\maketitle

\section*{Abstract}

Graph colorings have been of interest to mathematicians for a long time, but relatively recently, social scientists have also found them to be interesting tools for studying group behavior. In the last 20 years, scientists have begun to study how coloring problems can be solved by groups of individuals on a graph, which has led to new insights into network structure, group dynamics, and individual human behavior. Despite this newfound utility, the exact nature of these distributed coloring problems is not well-understood, and established mathematical tools like the chromatic polynomial miss the unique challenges that arise in these social problem-solving situations with limited information. In this paper, we provide a new framework for understanding these distributed problems by defining a new kind of graph coloring with particular relevance to consensus formation on networks, in which all vertices are trying to agree on a common color. These \emph{strict gridlock} colorings represent roadblocks to consensus where the group will not reach a uniform coloring using natural update processes. We describe a recurrence relation that provides an algorithm for counting these gridlocked colorings, which establishes a mathematical measure of how much a given graph hinders consensus in a group.

\section{Introduction}

To mathematicians, graph colorings can be interesting purely for their intriguing properties and puzzle-like nature. However, they have also proven useful for a surprisingly diverse body of problems from structural questions in Ramsey theory~\cite{Graham_Rothschild_Spencer_2013} to the coloring of maps in the famous four color theorem~\cite{Appel_Haken_1989}.

The traditional notion of a graph coloring, which we refer to as a \emph{proper} graph coloring, is a labeling of the vertices of a graph with $k$ colors so that no two adjacent vertices have the same color, and these mathematical objects have many interesting properties. However, this core idea has been expanded in many different ways. Perhaps the most obvious extension is to color the edges of a graph, instead of the vertices. 
After defining proper (vertex or edge) colorings, one could introduce additional conditions on the colorings to get even more restricted sets of colorings.
For example, equitable colorings require that all colors appear the same number of times~\cite{Lih_1998}.
Harmonious colorings require that for each pair of colors, at most one edge in the entire graphs connects vertices with those colors~\cite{Georges_1995}.
Rainbow colorings are colorings on the edges of a graph such that for every pair of vertices, there exists a path between them where no color appears more than once~\cite{Chartrand_Johns_McKeon_Zhang_2008}.
List colorings, unlike the other colorings above, assign each vertex a list of colors. In this case, the challenge is typically to find the shortest possible lists so that a proper coloring can be found by assigning each vertex colors from its list~\cite{Erdos_Rubin_Taylor_1979}.

\subsection{Proper Graph Colorings in the Social Sciences}
These graph coloring problems appear in a wide variety of applications. 
In particular, graph coloring problems have been used specifically to study the behavior and coordination of groups, with each participant required to select the color for a single vertex using limited information, and with color representing some behavior or action of the individual. This has been a fruitful approach and has yielded many empirical results linking group structure and coordination.
Some of the first empirical studies examined different network structures that made the problem more or less difficult~\cite{Kearns_Suri_Montfort_2006, Judd_Kearns_Vorobeychik_2010, McCubbins_Paturi_Weller_2009}.
Later work focused on the behavior of individuals embedded in the network, considering both random color choice~\cite{Shirado_Christakis_2017, Jones_Pauls_Fu_2021a} and altering the network structure~\cite{Chiang_Cho_Chang_2024}.
These problems, which exist in a social setting involving many decision-makers, make the problem more challenging than the traditional graph coloring problem. (Note that solving these graph problems in a distributed setting is different than on-line graph coloring, in which vertices must be colored one at a time by a central decision-maker~\cite{Lovasz_Saks_Trotter_1989}.)

Proper graph colorings can be thought of as modeling anti-coordination problems in groups, in which individuals must choose a different strategy or behavior from their neighbors. However, not every collective action problem is an anti-coordination problem. Other social challenges can also be modeled with (non-proper) graph coloring problems. As one example, Erikson and Shirado used a variation on proper graph colorings to study the division of labor on networks~\cite{Erikson_Shirado_2021}. In this problem, it is acceptable for a vertex to share a color with its neighbors as long as it is adjacent to all the available colors.

\subsection{Consensus and Uniform Colorings}

In other common scenarios, groups are tasked with reaching consensus on one of $k$ possible choices. This ranges from friends choosing a restaurant for dinner to a pack of animals choosing which a direction of travel~\cite{Couzin_Krause_Franks_Levin_2005} to jury verdicts of guilty or not guilty~\cite{Feddersen_Pesendorfer_1998}, each with their own complications. From a game theory perspective, all these problems fall into the broad category of coordination games, and can be further divided by the types of incentives that participants face into several well-studied games, including Matching Pennies~\cite{Belot_Crawford_Heyes_2013} or Battle of the Sexes~\cite{Robinson_Goforth_2012}. In the iterated setting, a natural but naive strategy is to choose the most common choice among all players in the previous round. 

However, when these games are played on networks and participants have limited knowledge of the group as a whole, additional challenges arise. Individuals are responsible for selecting a strategy to align with their neighbors without knowing how those choices affect the entire network. Instead of reaching consensus, naive and greedy strategies (for a rigorous description of these update rules and variants, see \cite{Jones_Pauls_Fu_2021b}) can lead to gridlocked states, in which all participants are making the ``best'' choice and the population stagnates without deciding on a common outcome. 

Such challenges can be modeled by another (non-proper) graph coloring problem: find a uniform coloring in which all vertices have the same color. Of course, this is trivial from a global perspective. But when solved in a social setting with multiple decision-makers and limited information, they become highly non-trivial and represent solutions to consensus formation problems. 
Jones and Christakis tasked individuals with finding a uniform coloring to study the role of leaders in the formation of consensus in networks~\cite{Jones_Christakis_2024}. Other studies examined how groups find uniform colorings (framed as consensus or voting) with unbalanced incentives and found that network topology shaped the group's ability to reach consensus and controlled if the majority or minority preference was selected~\cite{Judd_Kearns_Vorobeychik_2010, Kearns_Judd_Tan_Wortman_2009}. Despite this progress, there has been very little research studying these problems analytically. Related work in the continuous space like the DeGroot Model~\cite{Degroot_1974} misses the challenges that are present in the discrete space. 

A key component to understanding the uniform graph coloring problem, and therefore understanding consensus, is to define and study the states corresponding to gridlock, the dead ends in the search for a global solution to the consensus problem. To that end, we introduce two new related graph colorings: the locally optimal coloring and the strict gridlock coloring. Then we describe a recursive algorithm that counts these colorings, showing how graph structure can help or hinder the formation of consensus in networked groups.

\section{The Locally-Optimal Polynomial}

Let $G = (V,E)$ be a graph. A \vocab{locally-optimal $k$-coloring} is a function $c: V \to [k]$, thought of as assigning a color to each vertex, such that each $v \in V$ shares a color with a plurality of its neighbors. Here, we are using the word ``plurality'' to mean that among the neighbors of $v$, some color appears more often than any other color. As a trivial consequence, if a vertex $v$ has no neighbors, then no color can have a plurality among the neighbors of $v$.

A locally-optimal coloring is a \vocab{strict gridlock coloring} if two or more distinct colors are assigned to vertices; otherwise it is a \vocab{consensus coloring}. 

Given a graph $G$, let $\sLO_k(G)$ denote the set of locally-optimal $k$-colorings of $G$, and let
\begin{equation*}
    \LO_k(G) = \lvert \sLO_k(G) \rvert
\end{equation*}
be the cardinality of this set. We call $\LO_k(G)$ the \vocab{LO-polynomial} of $G$, which is justified by the following theorem.
\begin{theorem}\label{thm:polynomial}
    The number of locally-optimal $k$-colorings of $G$ is a polynomial in $k$.
\end{theorem}
The proof of \cref{thm:polynomial} appears later in this paper after we have developed a bit more machinery for reasoning about locally-optimal colorings.

We can also define the \vocab{SG-polynomial} $\SG_k(G)$ of a graph $G$ to be the number of colorings of $G$ which are strict gridlocks. Since each graph has exactly $k$ consensus colorings given $k$ total colors, $\SG_k(G)$ only differs from $\LO_k(G)$ by the absence of these colorings. In symbols,
\begin{equation*}
    \SG_k(G) = \LO_k(G) - k \,.
\end{equation*}
While $\SG_k(G)$ counts the ``interesting'' gridlocks, it turns out that the LO-polynomial has better properties, so most of our results in this paper will be phrased in terms of locally-optimal colorings. Here are some immediate examples of the sort of properties that $\LO_k(G)$ has.

\begin{proposition}\label{thm:complete}
    The LO-polynomial of a complete graph $K_n$ is $\LO_k(K_n) = k$.
\end{proposition}

\begin{proof}
    Assume for the purpose of contradiction that $K_n$ has a gridlocked coloring with $n_1$ vertices of color $c_1$, $n_2$ vertices of color $c_2$, and so on. Without loss of generality, assume $n_1 \le n_2 \le n_3 \le \cdots$. Pick some vertex $v$ with color $c_1$. The vertex $v$ must have $n_1 - 1$ neighbors of color $c_1$, but it must also have $n_2$ neighbors of color $c_2$. Since $n_1 - 1 \le n_2$, $c_1$ cannot be the plurality color among neighbors of $v$, which is a contradiction. Therefore, $K_n$ has no gridlocked colorings, which means the only locally-optimal colorings are the $k$ consensus colorings.
\end{proof}

The following proposition lists a few more properties of the $\LO$-polynomial, whose proofs follow from the theorems in the next section.

\begin{proposition}
\label{thm:properties}
    Let $G$ be a nonempty connected graph. The following are true:
    \begin{enumerate}[label=(\alph*)]
        \item The degree of $\LO_k(G)$ is the maximum number of distinct colors possible in a gridlocked coloring of $G$. 
        \item The leading coefficient of $\LO_k(G)$ is the number of unique ways to partition the vertices of $G$ into the maximum number of parts such that each vertex is in the same part as a plurality of its neighbors.
        \item The constant term of $\LO_k(G)$ is always zero.
    \end{enumerate}
\end{proposition}

\begin{corollary}\label{thm:minimal}
    If a connected graph $G$ has $n \ge 3$ vertices, then the degree of $\LO_k(G)$ is bounded above by $\lfloor \frac{n}{3} \rfloor$.
\end{corollary}

\begin{proof}
    We will show that any locally-optimal coloring of $G$ must assign each color to at least three vertices. By part (a) of \cref{thm:properties}, this gives us a bound on the degree of $\LO_k(G)$.
    
    Assume for the purpose of contradiction that some color $c$ is only assigned to one or two vertices. In the first case, if $c$ is only assigned to one vertex $v$, then by connectedness we know that $v$ has some other neighbor which must be assigned a different color. Therefore, $c$ cannot be the plurality color among the neighbors of $v$.
    
    In the second case, the color $c$ is only assigned to two vertices $v$ and $w$. Again, our connectedness assumption along with our assumption that $G$ has at least three vertices means that $v$ and/or $w$ must have another neighbor. Without loss of generality, assume that $v$ has another neighbor. Since this neighbor must have a color that isn't $c$, $c$ cannot be the majority color among the neighbors of $v$. Therefore, in both cases we get a contradiction, so each color must be assigned to at least three vertices.
\end{proof}

Given any degree $d$, one can realize this upper bound by constructing a connected graph $G$ with $3d$ vertices such that $\deg(\LO_k(G))=d$ as in 
\cref{fig:minimal}.

\begin{figure}[ht]
    \centering
    \includegraphics[width=0.5\linewidth]{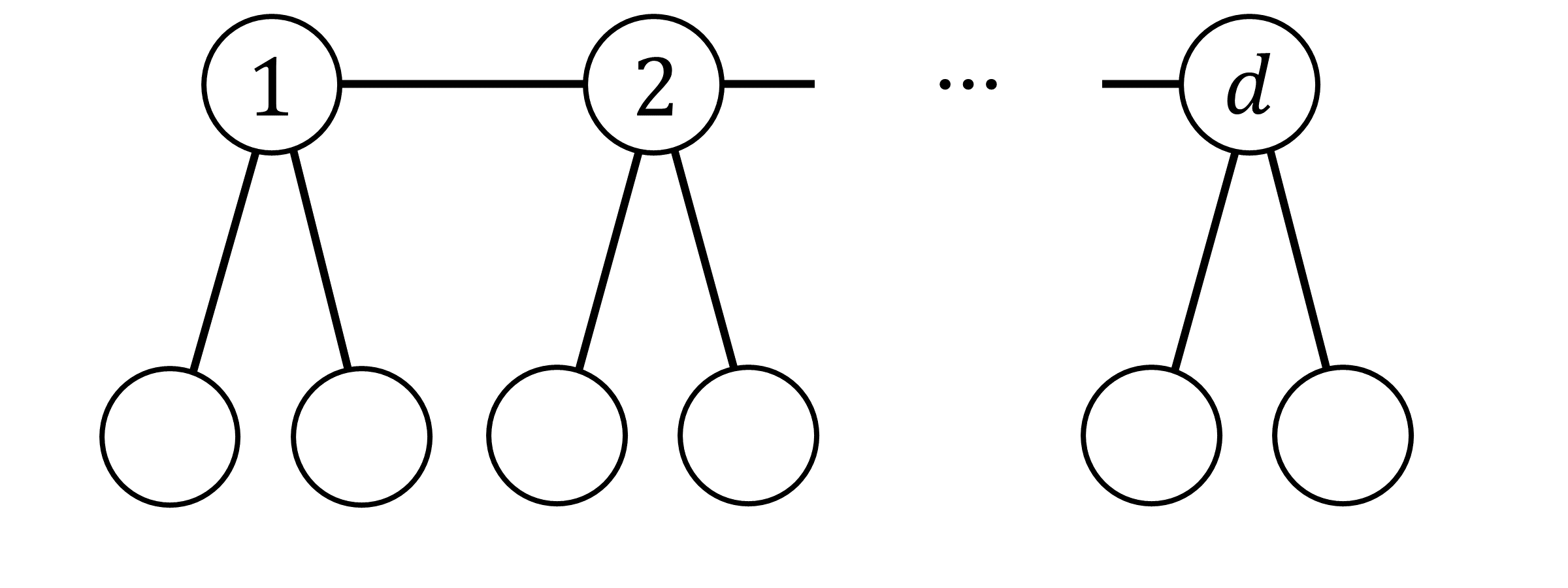}
    \caption{The edge-minimal connected graph with a specified $\LO$-polynomial degree $d$.}
    \label{fig:minimal}
\end{figure}

\section{Recurrence Algorithm}\label{sec:recurrence}

\subsection{Low-degree vertices}

In this section, we discuss a few properties of $\LO_k(G)$ involving low-degree vertices. We conclude by describing a recursive algorithm to compute the $\LO$-polynomial of graphs with maximum degree three entirely in terms of other such graphs.

\subsubsection{Degree one}

We refer to vertices of degree one as \vocab{leaves}, even if the graphs in question are not trees.

\begin{lemma}\label{thm:leaf}
    Let $G = (V,E)$ be a graph and let $v \in V$ be a leaf with unique neighbor $w$. Then any locally-optimal coloring $c$ of $G$ will satisfy $c(v) = c(w)$.
\end{lemma}

\begin{proof}
    Let $c$ be a locally-optimal coloring of $G$. Since $v$ only has one neighbor, the color $c(w)$ assigned to $v$ is by default the plurality color among the neighbors of $v$, and thus $c(v)$ must equal $c(w)$.
\end{proof}

\subsubsection{Degree two}

We may refer to vertices of degree two as \vocab{bivalent vertices}.

\begin{lemma}\label{thm:bivalent}
    Let $G = (V,E)$ be a graph and let $v \in V$ be a bivalent vertex with neighbors $u$ and $w$. Then any locally-optimal coloring $c$ of $G$ will satisfy $c(u) = c(v) = c(w)$.
\end{lemma}

\begin{proof}
    Let $c$ be a locally-optimal coloring of $G$. If $c(u) \not= c(w)$, then there is no plurality color among the neighbors of $v$, which contradicts our assumption that $c$ is locally-optimal. Therefore, $c(u) = c(w)$, and thus $c(u)$ is the plurality color among the neighbors of $v$, so $c(u) = c(v)$.
\end{proof}

One consequence of \cref{thm:bivalent} is that it lets us describe the effect of edge subdivision on $\LO_k(G)$.

\begin{corollary}\label{thm:subdivision}
Let $G = (V,E)$ be a graph, and let $uw$ be an edge of $G$. Consider the graph $G'=(V',E')$ obtained from $G$ by subdividing the edge $uw$, creating a new vertex $v$ and replacing $uw$ by two new edges $uv$ and $vw$. Then $\LO_k(G')$ counts the number of $k$-colorings of $G$ such that $u$ and $w$ have the same color.
\end{corollary}

\begin{proof}
Let $S \subseteq \sLO_k(G)$ be the set of locally-optimal $k$-colorings $c$ of $G$ such that $c(u) = c(w)$. We will describe a bijection between $S$ and the set $\sLO_k(G')$ of locally-optimal $k$-colorings of $G'$.

First, let $c \in S$ be a coloring of $G$ such that $c(u) = c(w)$. We can extend $c$ to a coloring $c'$ of $G'$ by letting $c'(a) = c(a)$ for all $a \in V$, and letting $c'(v) = c(u)$. This coloring is locally-optimal at every vertex $a \in V$ by assumption. This coloring is also locally optimal at $v$ since $v$ has the same color as both of its neighbors $u$ and $w$. Therefore, this extension $c'$ is a valid locally-optimal $k$-coloring of $G'$.

In the other direction, assume that $c'$ is a locally-optimal $k$-coloring of $G'$. Any such coloring of $G'$ must satisfy the condition that $c'(u) = c'(v) = c'(w)$ by \cref{thm:bivalent}. Therefore, restricting $c'$ to the vertices of $G$ gives us a coloring $c$ of $G$ such that $c(u) = c(w)$.
\end{proof}

Graphs whose vertices have sufficiently-many bivalent neighbors have particularly simple $\LO$-polynomials. We call a graph $G$ \vocab{bivalent-dense} if a majority of the neighbors of each non-bivalent vertex are bivalent. The following lemma tells us that bivalent-dense graphs have monomials for their LO-polynomials.

\begin{lemma}\label{thm:base-case}
    Let $G$ be a bivalent-dense graph, and let $I \subseteq G$ be the subgraph induced by the set of edges incident to at least one bivalent vertex. If $I$ contains $p$ connected components, then
    \begin{equation*}
        \LO_k(G) = \LO_k(I) = k^p \,.
    \end{equation*}
\end{lemma}

\begin{proof}
    Let $S$ be the set of connected components of the edge-induced subgraph $I$. Since each component is internally ``connected'' by bivalent vertices, each vertex within a connected component must have the same color as all the other members of its component. Therefore, $\LO_k(G)$ is bounded above by the number of functions $S \to [k]$, which is $k^p$. Furthermore, we can see that every function $f: S \to [k]$ gives us a valid locally-optimal coloring $c$ of $G$. Define the coloring $c$ induced by $f$ on each vertex $v$ to be such that $c(v)$ is the color that $f$ assigns to the component containing $v$. By assumption, a majority of the neighbors of each non-bivalent vertex are bivalent, so a majority of the neighbors of $v$ are in the same component as $v$. Therefore, any assignment of colors to the components $S$ is a valid locally-optimal coloring, so $\LO_k(G) = k^p$.
\end{proof}

\subsubsection{Degree three}

Next, we will use \cref{thm:subdivision} to prove a local recurrence relation for vertices of degree three, which we may refer to as \vocab{trivalent vertices}.

\begin{theorem}\label{thm:trivalent}
Let $G$ be a graph, and consider a trivalent vertex $v$ of $G$ with neighbors $a$, $b$, and $c$. Let $G_{ab}$ denote the graph obtained from $G$ by subdividing the edges $va$ and $vb$, and define the graphs $G_{ac}$, $G_{bc}$ and $G_{abc}$ analogously. Then
    \begin{equation*}
        \LO_k(G) = \LO_k(G_{ab}) + \LO_k(G_{ac}) + \LO_k(G_{bc}) - 2 \LO_k(G_{abc}) \,.
    \end{equation*}
\end{theorem}

\begin{proof}
In order for the vertex $v$ to have a locally-optimal color, at least two of its three neighbors must have the same color. Each choice of two or three neighbors to share colors gives a different set of locally-optimal colorings, and \cref{thm:subdivision} tells us that we can count them by subdividing those edges:
\begin{equation*}
    \sLO_k(G) = \sLO_k(G_{ab}) \union \sLO_k(G_{ac}) \union \sLO_k(G_{bc}) \,.
\end{equation*}
Here, we are implicitly identifying colorings of graphs obtained from $G$ by subdivision with colorings of $G$. The proof of \cref{thm:subdivision} justifies why these colorings are still valid when restricted to the vertices of $G$.

However, if an edge is not subdivided, the two vertices at the ends can still have the same color in a locally-optimal coloring. To count the true number of colorings in this union of three sets, we apply the inclusion-exclusion principle to avoid over-counting the the colorings in which all three neighbors have the same color as $v$:
\begin{align*}
    \card{\sLO_k(G)} &= \card{\sLO_k(G_{ab}) \union \sLO_k(G_{ac}) \union \sLO_k(G_{bc})} \\
    &= \card{\sLO_k(G_{ab})} + \card{\sLO_k(G_{ac})} + \card{\sLO_k(G_{bc})} - 2\card{\sLO_k(G_{abc})} \\
    \LO_k(G) &= \LO_k(G_{ab}) + \LO_k(G_{ac}) + \LO_k(G_{bc}) - 2 \LO_k(G_{abc}) \,. \qedhere
\end{align*}
\end{proof}

Note that, in each term on the right-hand side of the above equation, a majority of the neighbors of the vertex $v$ are bivalent. Therefore, iteratively applying \cref{thm:trivalent} at every vertex of $G$ results in a linear combination of LO-polynomials of bivalent-dense graphs. This establishes a recursive algorithm for computing $\LO_k(G)$ for graphs of maximum degree at most three.

\begin{figure}
    \centering
    \includegraphics[width=\linewidth]{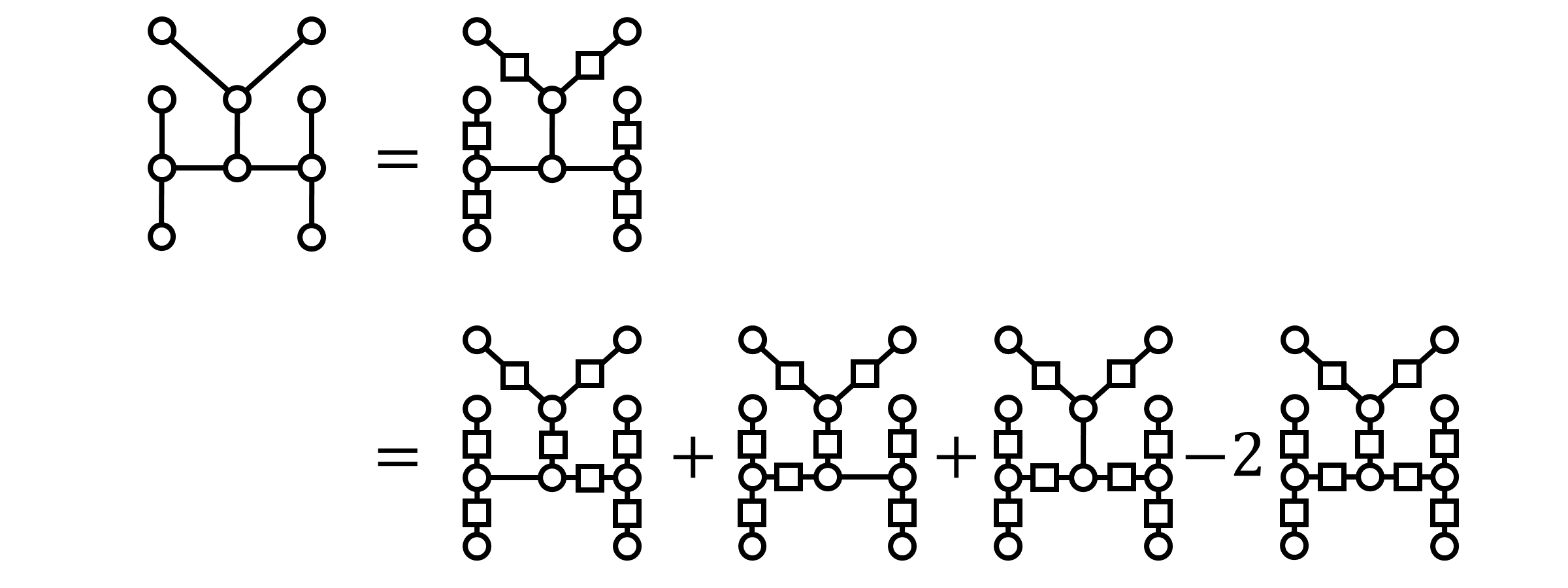}
    \caption{Demonstration of computing the locally-optimal polynomial for a graph with a low maximum degree. Bivalent vertices are represented by squares. First, we subdivide the edges connected to leaves, and then we apply \cref{thm:trivalent} on the central vertex. All graphs in the bottom row are bivalent dense and therefore have easily accessible LO polynomials of $k^2$ or $k$.}
    \label{fig:deg_three_recur}
\end{figure}

\subsection{General graphs}\label{sec:general-graphs}

First, we note that there is a straightforward generalization of \cref{thm:trivalent} to vertices of any degree $n \ge 3$.

\begin{theorem}\label{thm:any-degree}
    Let $G$ be a graph, and consider an $n$-valent vertex $v$ of $G$ with neighbor set $N(v) = \{w_1, w_2, \dots, w_n\}$. For any subset $S \subseteq N(v)$, let $G_S$ be the graph obtained from $G$ by subdividing each edge $vw$ for $w \in S$. Then
    \begin{equation*}
         \LO_k(G) = \sum_{j=2}^n (-1)^j (j-1) \sum_{\substack{S \subseteq N(v) \\ \card{S} = j}} \LO_k(G_S) \,.
    \end{equation*}
\end{theorem}

\begin{proof}
    By assumption that $n \ge 2$, $v$ must share a color with at least two neighbors in any locally-optimal coloring of $G$. Therefore, we can write the set of locally-optimal colorings of $G$ as a union of sets in which two of the edges incident to $v$ are subdivided:
    \begin{equation*}
        \sLO_k(G) = \bigunion_{\substack{S \subseteq N(v) \\ \card{S} = 2}} \sLO_k(G_S) \,.
    \end{equation*}
    As before, we are implicitly identifying colorings of graphs obtained from $G$ by subdivision with colorings of $G$. In order to translate the above into an equation about the $\LO$-polynomials, we can use an inclusion-exclusion argument to account for the overlap of $\sLO_k(G_S)$ for different choices of $S$.

    Consider the sum of the $\LO$-polynomials corresponding to each way to choose two edges to subdivide; this quantity counts all valid colorings of $G$, but also over-counts colorings in which three neighbors have the same color as $v$:
    \begin{equation*}
        \LO_k(G) = \sum_{\substack{S \subseteq N(v) \\ \card{S} = 2}} \LO_k(G_S) - \dots \,.
    \end{equation*}
    Note that if exactly three neighbors share a color with $v$, then these colorings are triply-counted. We can partially correct this issue by removing two copies of $\LO_k(G_S)$ for each choice of $\card{S} = 3$:
    \begin{equation*}
        \LO_k(G) = \sum_{\substack{S \subseteq N(v) \\ \card{S} = 2}} \LO_k(G_S) - 2 \sum_{\substack{S \subseteq N(v) \\ \card{S} = 3}} \LO_k(G_S) + \dots \,.
    \end{equation*}
    However, now we are counting each coloring with 4 same-colored neighbors a total of $\binom{4}{2} - 2 \binom{4}{3} = -2$ times, so we should add these colorings back in three times:
    \begin{equation*}
        \LO_k(G) = \sum_{\substack{S \subseteq N(v) \\ \card{S} = 2}} \LO_k(G_S) - 2 \sum_{\substack{S \subseteq N(v) \\ \card{S} = 3}} \LO_k(G_S) + 3 \sum_{\substack{S \subseteq N(v) \\ \card{S} = 4}} \LO_k(G_S) - \dots \,.
    \end{equation*}
    In general, if a locally-optimal $k$-coloring of $G$ assigns $v$ the same color as exactly $j$ of its neighbors, then the above alternating sum will count this coloring
    \begin{equation*}
        \sum_{i=2}^{j-1} (-1)^i(i-1) \binom{j}{i}
    \end{equation*}
    times in previous terms of the sum. This sum has the simpler closed form
    \begin{equation*}
        1-(-1)^j(j-1) \,.
    \end{equation*}
    Therefore, to correct for this, we should add back in $\LO_k(G_S)$ exactly 
    \begin{equation*}
        1 - (1-(-1)^j(j-1)) = (-1)^j(j-1)
    \end{equation*}
    times in each term of the right-hand side.
\end{proof}

\begin{example}\label{ex:degree-four}
    As before, let $G_{abc\dots}$ denote the graph obtained from $G$ by subdividing the edges $va$, $vb$, $vc$, and so on. When $n=4$, \cref{thm:any-degree} gives us the formula
    \begin{align*}
        \LO_k(G) &= \big(\LO_k(G_{ab}) \begin{aligned}[t]
            &+ \LO_k(G_{ac}) + \LO_k(G_{ad}) \\
            &+ \LO_k(G_{bc}) + \LO_k(G_{bd}) + \LO_k(G_{cd})\big)
        \end{aligned} \\
        &- 2\big(\LO_k(G_{abc}) + \LO_k(G_{abd}) + \LO_k(G_{acd}) + \LO_k(G_{bcd})\big) \\
        &+ 3\LO_k(G_{abcd})
    \end{align*}
\end{example}

\begin{figure}
    \centering
    \includegraphics[width=0.75\linewidth]{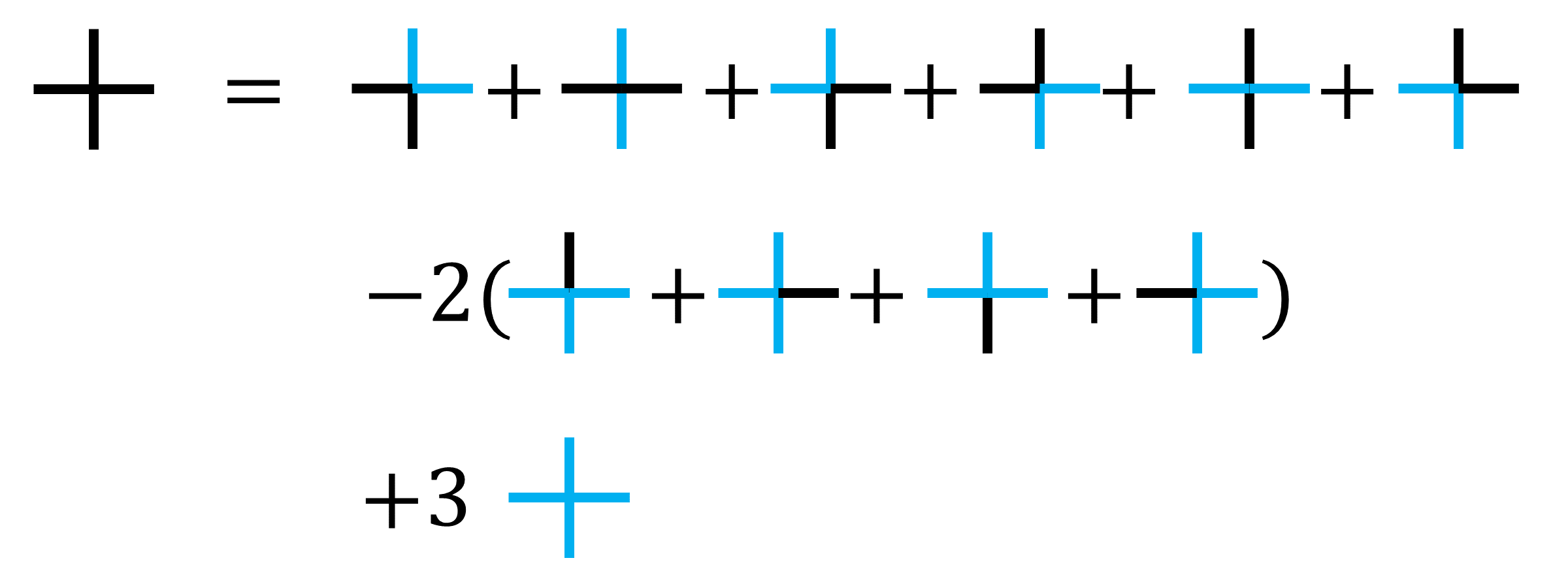}
    \caption{Demonstration of \cref{thm:any-degree} where $v$ is a degree four vertex. Each symbol shows the four edges connected to $v$ and represents the LO-polynomial of the graph. Initially, we have no information about which neighbors $v$ shares a color with, but \cref{thm:any-degree} tells us how we can express this LO-polynomial in terms of other LO-polynomials where certain edges are subdivided, represented by blue.}
    \label{fig:degree_four_recur_1}
\end{figure}

Note that unlike our previous theorem for trivalent vertices, the above \cref{thm:any-degree} does not produce bivalent-dense graphs when $n \ge 4$. Each term in the right-hand side is only guaranteed to have at least two bivalent neighbors for the chosen vertex $v$. Therefore, more simplification is needed to fully compute the LO-polynomial.

To assist in calculating the LO-polynomial of general graphs, we will allow our graph to contain \vocab{non-voting edges}. These are edges that do not count when determining the plurality color among the neighbors of a vertex, but force their endpoints to have the same color in any LO-coloring. We can think of these as extra data on top of the usual graph structure. These non-voting edges do not count towards the degree of a vertex, and we ignore these edges for the purposes of determining whether or not a graph is bivalent dense; they are simply added to restrict the possible LO-colorings of a graph. With these non-voting edges, we can now state the other half of the recurrence for calculating LO-polynomials.

\begin{theorem}\label{thm:recurrence}
    Let $G$ be a graph, and consider an $n$-valent vertex $v$ of $G$ with $b$ bivalent neighbors, where $2 \le b \le \frac{n}{2}$. Let $W$ denote the set of all non-bivalent neighbors of $v$, and let $\powerset(W, b)$ denote the set of all $b$-element subsets of $W$. Let $\ell = n - 2b + 1$. Given a set $S$ of non-bivalent neighbors of $v$ and a set $\mathcal{T}$ of sets of non-bivalent neighbors of $v$, let $G^\ell_{S,\Tset}$ denote the graph obtained from $G$ by adding $\ell$ leaves to the vertex $v$, subdividing each edge $vw$ for $w \in S$, and adding non-voting edges $t_1t_2$ for all $t_1,t_2 \in T$ for each set $T \in \Tset$. Then
    \begin{align*}
        \LO_k(G) &= \LO_k(G^\ell_{\emptyset,\emptyset}) \\
        &- \left( \sum_{r=1}^{n-b + 2^{n-b}} (-1)^{r+1} \sum_{\substack{W' \subseteq W \\
        \Tset \subseteq \powerset(W, b) \\ \card{W'} + \card{\Tset} = r}} \LO_k(G^\ell_{W', \Tset}) \right) \\
        &+ \sum_{j=1}^{n-b} (-1)^{j+1} \sum_{\substack{W' \subseteq W \\ \card{W'} = j}}\LO_k(G^0_{W',\emptyset}) \,.
    \end{align*}
\end{theorem}

\begin{proof}
    We will first prove the analogous equality involving sets of LO-colorings, and then translate this into a statement about LO-polynomials. To begin, note that our chosen vertex $v$ has $n$ neighbors, $b$ of which must have the same color as $v$ in any LO-coloring of $G$. Since $b < \frac{n}{2}$, we know that these neighbors are not sufficient for a majority, since that would require $\ceiling{\frac{n+1}{2}}$ same-color neighbors. However, if we add $\ell = n - 2b + 2$ leaves to the vertex $v$, then in the new graph $v$ has degree $2n - 2b + 2 = 2(n-b+1)$, and has $n-b+2 = (n-b+1)+1$ edges that can be subdivided, which constitute a majority of the neighbors of $v$. Therefore, we are interested in the set
    \begin{align*}
        \sLO_k(G^\ell_{\emptyset,\emptyset}) \,.
    \end{align*}

    Once we add $\ell$ leaves to $v$, we have converted LO-colorings of $G$ in which $v$ shares a color with exactly $b$ neighbors into colorings in which $v$ shares a color with a majority of its neighbors, but we have also introduced many new colorings of $G^\ell$ that are not valid LO-colorings of $G$. Specifically, any colorings of $G^\ell$ in which $b$ or more neighbors in $W$ share a color with each other (can be connected by non-voting edges) but do not share a color with $v$ may not correspond to LO-colorings of $G$. Therefore, we would like to remove all colorings of $G^\ell$ in which some set $T$ of $b$ distinct neighbors are connected by non-voting edges from the colorings of $G^\ell$; this set is denoted
    \begin{align*}
        \bigunion_{\substack{T \subseteq W \\ \card{T} = b }} \sLO_k(G^\ell_{\emptyset, \{T\}}) \,.
    \end{align*}
    We also subtract all colorings of $G^\ell$ in which more than $b$ edges are subdivided, which constitute the set
    \begin{align*}
        \bigunion_{w \in W} \sLO_k(G^\ell_{\{w\},\emptyset}) \,.
    \end{align*}
    We then add back these colorings as long as they are also valid colorings of $G$ (without the additional leaves):
    \begin{align*}
    \bigunion_{w \in W} \sLO_k(G^0_{\{w\},\emptyset}) \,.
    \end{align*}
    Putting all this together, we get the equation
    \begin{align*}
        \sLO_k(G) &= \sLO_k(G^\ell_{\emptyset, \emptyset}) \\ 
        &\smallsetminus \left(\bigunion_{w \in W} \sLO_k(G^\ell_{\{w\},\emptyset}) \quad \union \bigunion_{\substack{T \subseteq W \\ \card{T} = b }} \sLO_k(G^\ell_{\emptyset, \{T\}})\right) \\
        &\union \left(\bigunion_{w \in W} \sLO_k(G^0_{\{w\},\emptyset})\right) \,.
    \end{align*}
    
    To count the number of elements in $\mathcal{LO}_k(G)$, we simply apply the inclusion-exclusion principle to this set equality. The index $r$ in the theorem counts the number of elements of
    \begin{align*}
        \bigunion_{w \in W} \sLO_k(G^\ell_{\{w\},\emptyset}) \quad \union \bigunion_{\substack{T \subseteq W \\ \card{T} = b }} \sLO_k(G^\ell_{\emptyset, \{T\}})
    \end{align*}
    that we are intersecting, as shown in the requirement $\card{W'} + \card{\Tset} = r$. Every element of $W'$ represents an edge attached to the central vertex $v$ that can be subdivided, and every set inside $\Tset$ is a collection of $b$ peripheral vertices that are connected by non-voting edges.
    
    The second nested sum counts colorings where no leaves are added, but an additional $j$ of $v$'s edges are subdivided, once again applying the inclusion-exclusion principle to the union
    \begin{equation*}
        \bigunion_{w \in W} \sLO_k(G^0_{\{w\},\emptyset}) \,. \qedhere
    \end{equation*}
\end{proof}

While \cref{thm:recurrence} is rather complicated, we can get more intuition about what is going on by considering specific small values of $n$.

\begin{example}\label{ex:degree-four-step-two}
    Let $G_{ab}$ denote the graph obtained from $G$ by subdividing the edges $va$ and $vb$, and let $G_{ab \mid cd}$ denote the graph obtained from $G_{ab}$ by adding a non-voting edge $cd$. When $n=4$ and $b=2$, \cref{thm:recurrence} specializes to
    \begin{align*}
        \LO_k(G_{ab}) &= \LO_k(G^1_{ab}) \\
        &- \big( \LO_k(G^1_{abc}) + \LO_k(G^1_{abd}) + \LO_k(G^1_{ab\mid cd})\big) \\
        &+ \big( \LO_k(G^1_{abcd}) +  \LO_k(G^1_{abc\mid cd}) + \LO_k(G^1_{abd\mid cd})\big) \\
        &- \LO_k(G^1_{abcd\mid cd}) \\
        &+ \big(\LO_k(G_{abc}) + \LO_k(G_{abd})\big) \\
        &- \LO_k(G_{abcd}) \,.
    \end{align*}
Many of these terms actually represent the same set of colorings of $G$. For example, if all four edges are subdivided, then the addition of an extra leaf does not change the set of locally-optimal colorings, so they are the same. Similarly, if a non-voting edge connects two neighbors, and the edge to one neighbor is subdivided, then all three vertices must have the same color so both edges could be subdivided and the non-voting edge removed without changing the number of locally-optimal colorings. Canceling like terms gives us the much simpler equation
    \begin{equation*}
        \LO_k(G_{ab}) = \LO_k(G^1_{ab}) - \LO_k(G^1_{ab\mid cd}) + \LO_k(G_{abcd}) \,.
    \end{equation*}
\end{example}

\begin{figure}
    \centering
    \includegraphics[width=0.75\linewidth]{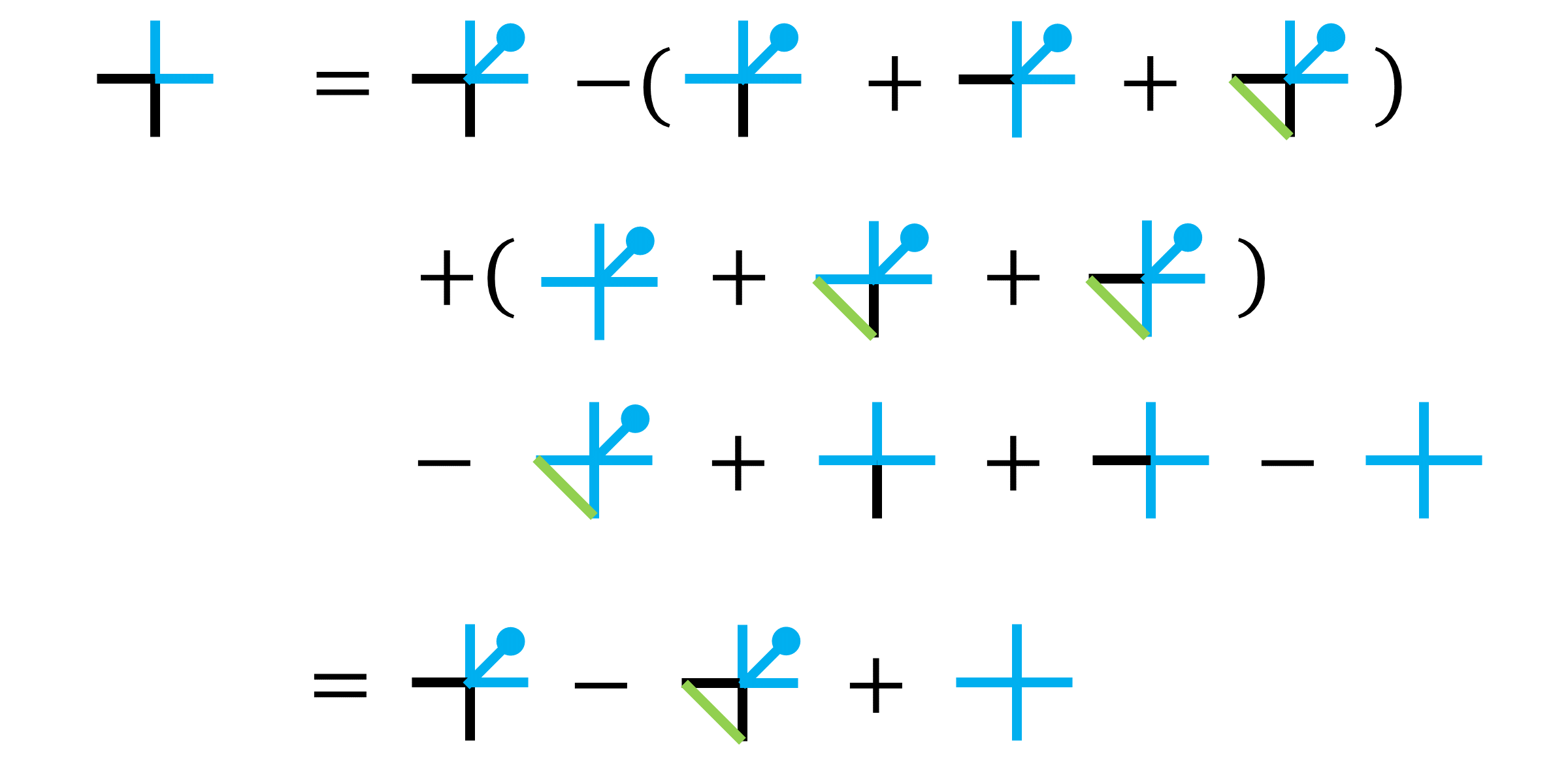}
    \caption{Demonstration of \cref{ex:degree-four-step-two} with degree 4. Like \cref{fig:degree_four_recur_1}, blue edges represent subdivided edges. An additional leaf in blue can be seen in the seven graphs in the top two rows. Non-voting edges, shown in green, connect vertices that have the same color but do not contribute towards local-optimality.}
    \label{fig:degree_four_recur_2}
\end{figure}

\begin{example}
    We can combine \cref{ex:degree-four} and \cref{ex:degree-four-step-two} to get the full analogue of \cref{thm:trivalent} for degree four vertices:
    \begin{align*}
        \LO_k(G) &= \big(\LO_k(G^1_{ab}) \begin{aligned}[t]
            &+ \LO_k(G^1_{ac}) + \LO_k(G^1_{ad}) \\
            &+ \LO_k(G^1_{bc}) + \LO_k(G^1_{bd}) + \LO_k(G^1_{cd})\big)
        \end{aligned} \\
        &- \big(\LO_k(G^1_{ab\mid cd}) 
        \begin{aligned}[t]
            &+\LO_k(G^1_{ac\mid bd}) + \LO_k(G^1_{ad\mid bc}) \\
            &+ \LO_k(G^1_{bc\mid ad}) + \LO_k(G^1_{bd\mid ac}) + \LO_k(G^1_{cd\mid ab})\big)
        \end{aligned} \\
        &+ 6\LO_k(G_{abcd}) \\
        &- 2\big(\LO_k(G_{abc}) + \LO_k(G_{abd}) + \LO_k(G_{acd}) + \LO_k(G_{bcd})\big) \\
        &+ 3\LO_k(G_{abcd}) \,.
    \end{align*}
    Combining the two occurrences of $\LO_k(G_{abcd})$ gives us the (slightly) shorter formula:
    \begin{align*}
        \LO_k(G) &= \big(\LO_k(G^1_{ab}) \begin{aligned}[t]
            &+ \LO_k(G^1_{ac}) + \LO_k(G^1_{ad}) \\
            &+ \LO_k(G^1_{bc}) + \LO_k(G^1_{bd}) + \LO_k(G^1_{cd})\big)
        \end{aligned} \\
        &- \big(\LO_k(G^1_{ab\mid cd}) 
        \begin{aligned}[t]
            &+\LO_k(G^1_{ac\mid bd}) + \LO_k(G^1_{ad\mid bc}) \\
            &+ \LO_k(G^1_{bc\mid ad}) + \LO_k(G^1_{bd\mid ac}) + \LO_k(G^1_{cd\mid ab})\big)
        \end{aligned} \\
        &- 2\big(\LO_k(G_{abc}) + \LO_k(G_{abd}) + \LO_k(G_{acd}) + \LO_k(G_{bcd})\big) \\
        &+ 9\LO_k(G_{abcd}) \,.
    \end{align*}
    Note that every graph appearing on the right-hand side is now ``closer'' to being bivalent dense (compared to $G$), in that a majority of the neighbors of $v$ have the same color as $v$. This indicates that we have indeed made progress towards being able to apply \cref{thm:base-case}.
\end{example}

\begin{figure}
    \centering
    \includegraphics[width=0.75\linewidth]{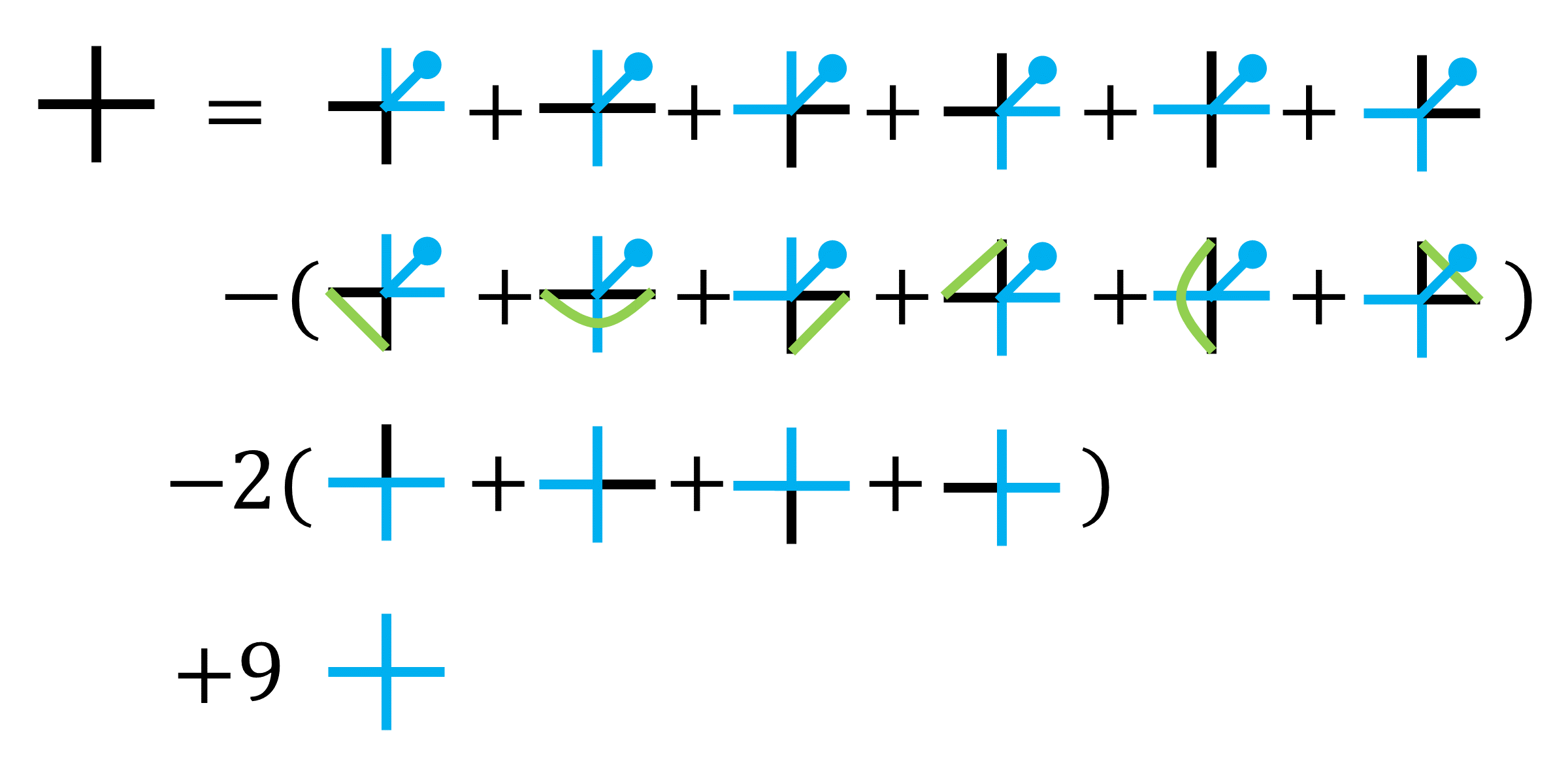}
    \caption{The complete recursion on four edges.}
    \label{fig:degree_four_complete}
\end{figure}

At this point, we have all the necessary pieces to build a recursive algorithm for computing $\LO_k(G)$:
\begin{enumerate}
    \item First, apply \cref{thm:any-degree} centered at each vertex of $G$ to express $\LO_k(G)$ as a linear combination of LO-polynomials of graphs with at least two same-color bivalent neighbors for every vertex.
    \item Next, apply \cref{thm:recurrence} at every vertex of $G$ to express $\LO_k(G)$ as a linear combination of LO-polynomials of bivalent-dense graphs.
    \item Finally, apply \cref{thm:base-case} to express $\LO_k(G)$ as a linear combination of powers of $k$.
\end{enumerate}
An implementation of this algorithm can be found in our code\footnote{\url{https://github.com/MattJonesMath/Strict_Gridlock_Graph_Colorings}}, linked in the footnote. We do not analyze the exact computational complexity of the algorithm in this paper, but the number of terms in \cref{thm:recurrence} grows exponentially with $n$ (both the number of values of $r$ and $\mathcal T \subseteq \mathcal P(W,b)$). While the algorithm we present is not optimized for speed, it is reasonable to assume that no polynomial time algorithms can exist, given the computational complexity of other graph coloring problems~\cite{Garey_Johnson_2009}.

One notable observation about the algorithm is that the run time appears to depend rather heavily on the \emph{order} in which vertices are visited by the recurrence. For example, if a graph contains one vertex of high degree and many vertices of low degree, we found that it was often much more efficient to apply \cref{thm:recurrence} at the low degree vertices first. This is likely because the number of terms in the recurrence equation balloons rapidly when the degree of the vertex increases. By starting with low degree vertices (where the recurrence is not too expensive), many edges attached to high degree vertices will already be subdivided, thus significantly reducing $n-b+2^{n-b}$, the number of graphs in the recurrence relation, for the high degree vertices.

This algorithm finally lets us prove that $\LO_k(G)$ is indeed a polynomial in $k$ for any graph $G$.

\begin{proof}[Proof of \cref{thm:polynomial}]
    To compute $\LO_k(G)$ for any graph $G$, we can first apply \cref{thm:any-degree} at every vertex to guarantee two subdivided edges at every vertex, then repeatedly apply \cref{thm:recurrence} to go from two subdivided edges to a majority of edges being subdivided. This lets us express $\LO_k(G)$ in terms of the LO-polynomials of graphs satisfying the bivalent dense hypothesis of \cref{thm:base-case}. Since these graphs have LO-polynomial $k^p$ for some $p$, this means that $\LO_k(G)$ is a polynomial for any $G$.
\end{proof}

\subsection{Set Partitions}
\label{sec:partitions}

While our previous results focused on graphs with subdivided edges and additional non-voting edges, we can also approach the LO-polynomial from the perspective of set partitions of the vertices. A \vocab{locally-optimal partition} $\pi$ of a graph $G = (V,E)$ is a set partition of the vertex set $V$ such that each vertex has a plurality of neighbors in the same part. We write $\pi \partitions G$ to denote that $\pi$ is a locally-optimal partition of $G$.

Locally-optimal partitions are like locally-optimal colorings where we have forgotten the exact colors and only remember the groupings of vertices. More precisely, to any locally-optimal coloring $c: V \to [k]$ we can associate the partition 
\begin{equation*}
    \pi_c = \{ c^{-1}(i) \}_{i \in [k]} \,.
\end{equation*}
In the opposite direction, each locally-optimal partition $\pi$ corresponds to
\begin{equation*}
    k^{\underline{\card{\pi}}} = k(k-1)(k-2)\dots(k-\card{\pi}+1)
\end{equation*}
locally-optimal $k$-colorings in this way. This gives us another formula for the LO-polynomial of a graph $G$, analogous to a well-known one for the chromatic polynomial involving a sum over independent sets in $G$.

\begin{theorem}\label{thm:partition}
Let $G$ be a graph. Then
    \begin{equation*}
        \LO_k(G) = \sum_{\pi \partitions G} k^{\underline{\card{\pi}}} \,.
    \end{equation*}
\end{theorem}

\begin{proof}[Another proof of \cref{thm:polynomial}]
    \cref{thm:partition} above expresses $\LO_k(G)$ as a sum of powers of $k$, and therefore proves that $\LO_k(G)$ is a polynomial.
\end{proof}

\begin{proof}[Proof of \cref{thm:properties}] Each of the following parts of \cref{thm:properties} follows as a corollary of \cref{thm:partition} above.
    \begin{enumerate}[label=(\alph*)]
        \item Gridlocked colorings of $G$ with the maximum number of colors possible correspond to partitions with the highest $\card{\pi}$, which correspond to the highest degrees of the polynomials $k^{\underline{\card{\pi}}}$. Therefore, the degree of $\LO_k(G)$ is the maximum number of parts possible in a gridlocked coloring of $G$.
        \item The leading coefficient of $\LO_k(G)$ is the number of unique LO-partitions $\pi$ with maximum $\card{\pi}$.
        \item Since it is not possible to partition the vertices of $G$ into zero parts, it is not possible for $\card{\pi}=0$. Therefore, the constant term of $\LO_k(G)$ is always zero. \qedhere
    \end{enumerate}
\end{proof}

Much of our previous work in this paper can be phrased in terms of partitions of the vertices of a graph. In particular, the algorithms in \cref{sec:general-graphs} are implemented in our code using set partitions.

\section{Strict Gridlock vs Group Structure}

\begin{figure}
    \centering

    \begin{subfigure}{0.45\textwidth}
      \includegraphics[width = \textwidth]{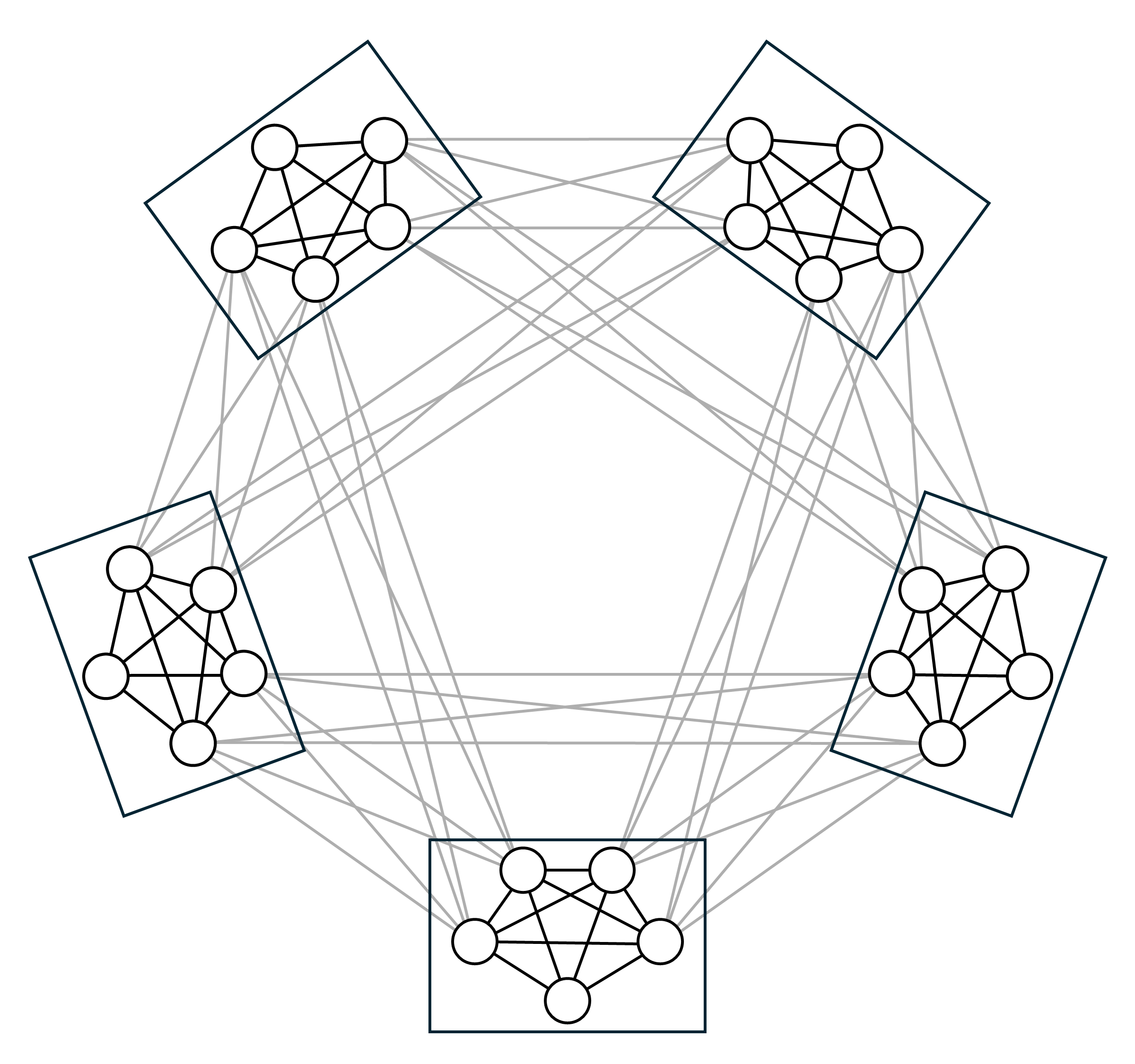}
      \caption{}
      \label{fig:G7}
    \end{subfigure}\hfill
    \begin{subfigure}{0.45\textwidth}
      \includegraphics[width = \textwidth]{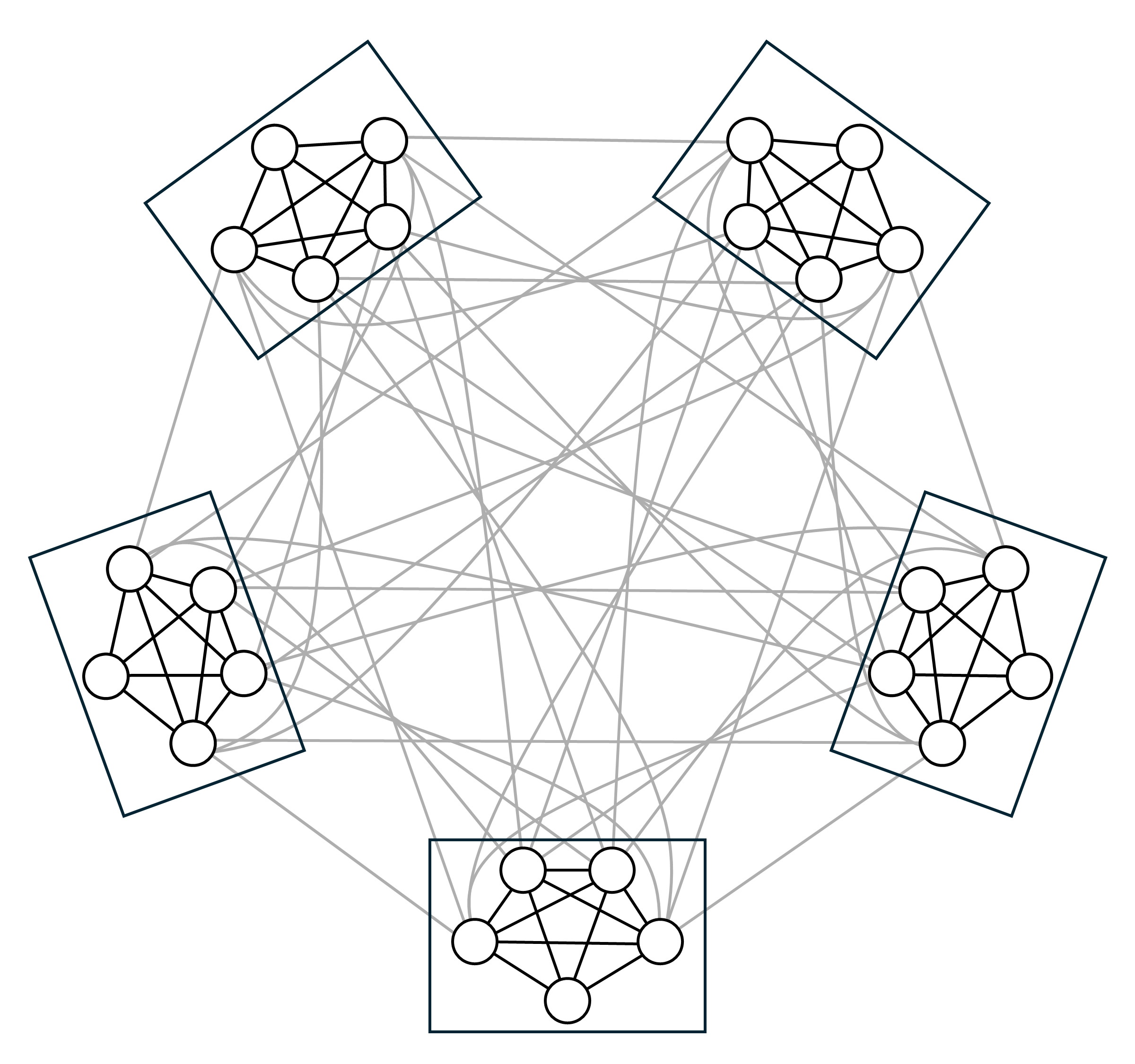}
      \caption{}
      \label{fig:G9}
    \end{subfigure}
    
\caption{Two graphs with similar group structure but different strict gridlock properties. Due to differences in the placement of between-clique edges, the graph on the left is less susceptible to gridlock than the graph on the right.}
    \label{fig:group_structure_graphs}
\end{figure}

Consider the two graphs shown in \cref{fig:group_structure_graphs}.
These graphs have extremely similar group structure; both have 5 fully-connected cliques of five vertices. Each clique has one ``exterior vertex'' with no between-clique edges and four ``interior vertices'' that have four in-clique and four between-clique edges. In both graphs, there are four edges connecting any two cliques, and so these graphs are largely indistinguishable to modularity-based community detection methods. The primary difference between the two is that in \cref{fig:G7}, the four between-clique edges connect two pairs of vertices, while in \cref{fig:G9}, the interior vertices all have one edge to each of the four other cliques.

Because of these differences, they have different SG polynomials. These graphs are too large to compute the strict gridlock polynomials using the recursion algorithm, but the graphs have enough structure that they can be described combinatorially. In \cref{sec:sg_poly_proof} of the appendix, we show that the SG polynomial of the graph in \cref{fig:G7} is $k^5 - 5k^4 + 10k^3 - 10k^2 + 4k$ and the SG polynomial of the graph in \cref{fig:G9} is $k^5 - 5k^2 + 4k$. For all values of $k$, the graph in \cref{fig:G9} has more strict gridlocked colorings than the graph in \cref{fig:G7}. For example, if $k = 2$, \cref{fig:G9} has 20 strict gridlocked colorings, while \cref{fig:G7} has none. 

These examples show that the SG polynomial and group structure can be somewhat decoupled. While group structure can drive gridlock, there are subtle network structures that can hinder or help the spread of information through a network without changing the communities inside the network.

\section{Conclusion}

The LO and SG polynomials are exciting new tools that expand the utility of graph theory methods to study group processes like consensus. Specifically, the SG polynomial can grant new insights into how network structure can subtly promote or discourage consensus in a network. Traditional proper graph coloring techniques have been used to examine the role of network structure in the past, but these typically count proper colorings, which represent solutions to a different class of social challenges. Community detection and clustering methods have also been used to study gridlock~\cite{Zhang_Friend_Traud_Porter_Fowler_Mucha_2008} but as we have shown, these are ill-fitting tools that miss the subtleties that the more bespoke SG polynomial can identify.

This work expands on previous work studying graph colorings and coordination problems analytically~\cite{Jones_Pauls_Fu_2021a, Jones_Pauls_Fu_2021b}, and there are many different paths along which this work can be extended.

First, the process of reaching consensus in a group is not static, but rather a dynamic process as the population attempts to solve the problem iteratively. There are colorings that are not locally-optimal but that can never reach a consensus coloring by choosing locally-optimal colors. In this work, we avoid those dynamics by only looking at fixed points in this process (hence ``strict'' gridlock), but this analysis can certainly be expanded. Identifying such ``weakly'' gridlocked states (see \cref{fig:nonstrict_gridlock} for an example) is significantly more challenging, because it requires thinking of the entire stochastic process instead of looking at a single point in time combinatorially, but we believe that at least some of these states can be identified and counted just like the strict gridlocked colorings.

\begin{figure}
    \centering
    \includegraphics[width=0.5\linewidth]{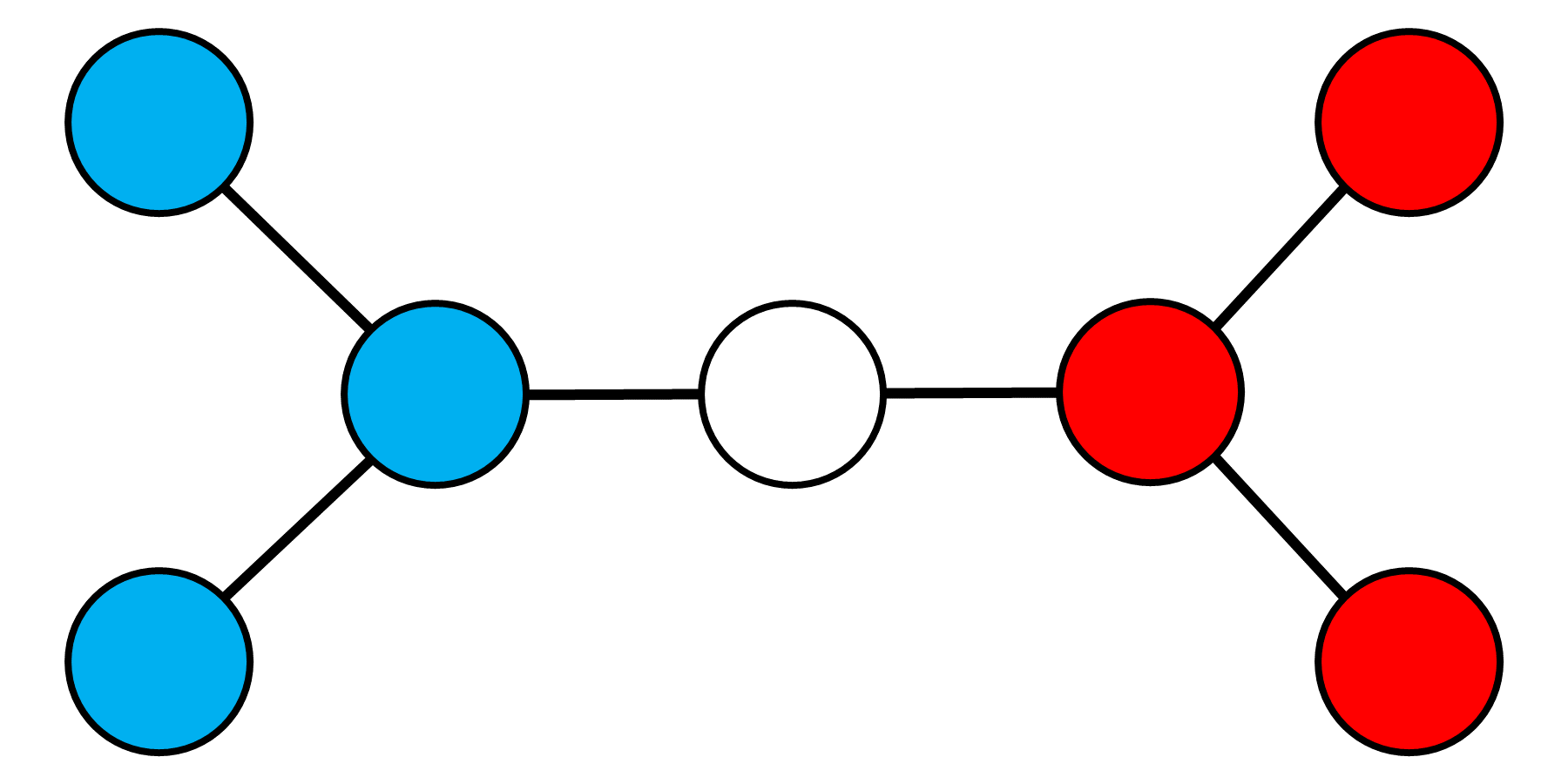}
    \caption{An example of non-strict gridlock on two colors. The central vertex will never have a locally-optimal choice, because the two neighbors are fixed on opposing colors, but this state will also never reach a uniform coloring.}
    \label{fig:nonstrict_gridlock}
\end{figure}

However, the SG polynomial can still, without any further development, be used to study group behavior and consensus formation. Applications could include an analysis of decision-making legislative bodies like the United States Congress, comparing networks (and their gridlocked colorings) with their legislative outputs, or studying the evolution of social networks in social animals or pre-modernized human groups. 

Another avenue of exploration is to continue studying gridlock when searching for a proper coloring. In~\cite{Jones_Pauls_Fu_2021b}, uniform and proper colorings are thought of as solutions to two different types of games; coordination and anti-coordination, respectively. These two classes of games are qualitatively very different, and yet they share some of the same challenges. In fact, ~\cite{Jones_Pauls_Fu_2021b} showed that these two games are equivalent when using two colors on bipartite graphs, in the sense that finding a uniform coloring is just as probable as finding a proper coloring. Combining the recursive algorithm in this paper with the isomorphism defined in \cite{Jones_Pauls_Fu_2021b} immediately categorizes all of the colorings that are strictly gridlocked in the search for a proper 2-coloring of a bipartite graph.

Unfortunately, this will certainly not hold for graphs with chromatic numbers of three or greater, and the relationship between coordination and anti-coordination is not fully understood. In the future, a similar approach to the recursive algorithm presented in this work could be used to study colorings that are gridlocked when the group is trying to find a proper coloring.

Finally, the LO and SG colorings are built around the simple update rule in which individuals select the plurality color among their neighbors. Adjustments to this strategy, perhaps by adding random noise or behavior~\cite{Shirado_Christakis_2017,Jones_Pauls_Fu_2021a} can help alleviate these issues but also require careful application to not backfire and prevent the group from settling into a global solution. Additionally, other kinds of human behavior like stubbornness have also been shown to improve outcomes~\cite{Kearns_Judd_Tan_Wortman_2009}. Incorporating more sophisticated behavior and update rules could improve the group's performance in solving the uniform coloring problem but would require a stochastic process framework instead of the simpler combinatorial approach used here. 

\bibliography{refs}

\begin{appendices}
\section{Computing the SG Polynomials of the graphs in \texorpdfstring{\cref{fig:group_structure_graphs}}{Figure 6}}\label{sec:sg_poly_proof}

\begin{lemma}\label{thm:consistent_color}
    In both graphs, each $5$-clique must be monochromatic in any locally-optimal coloring. 
\end{lemma}

\begin{proof}
    To start, choose any $5$-clique and consider the exterior (degree 4) vertex. Suppose, without loss of generality, that the vertex is blue. There are three cases: all five vertices in the clique are blue, three of the interior (degree 8) vertices are blue and one is another color (red), or two interior vertices are blue, one is red, and one is a third color (green). The case where the interior vertices come in two distinct pairs is not allowed because the exterior vertex in that clique will not have a locally-optimal choice.

    First, observe that in \cref{fig:G7}, interior vertices come in pairs, and each pair has the exact same set of neighbors. Therefore, in a locally-optimal coloring, these vertices must have the same color, which rules out every case except where all five interior vertices are blue.

    For \cref{fig:G9}, first consider the case where three interior vertices are blue and the last is red. This vertex has four blue neighbors and eight neighbors total. This coloring cannot be locally-optimal because this vertex cannot have more red than blue neighbors, but they are colored red.

    \begin{figure}
        \centering
    
        \begin{subfigure}{0.45\textwidth}
          \includegraphics[width = \textwidth]{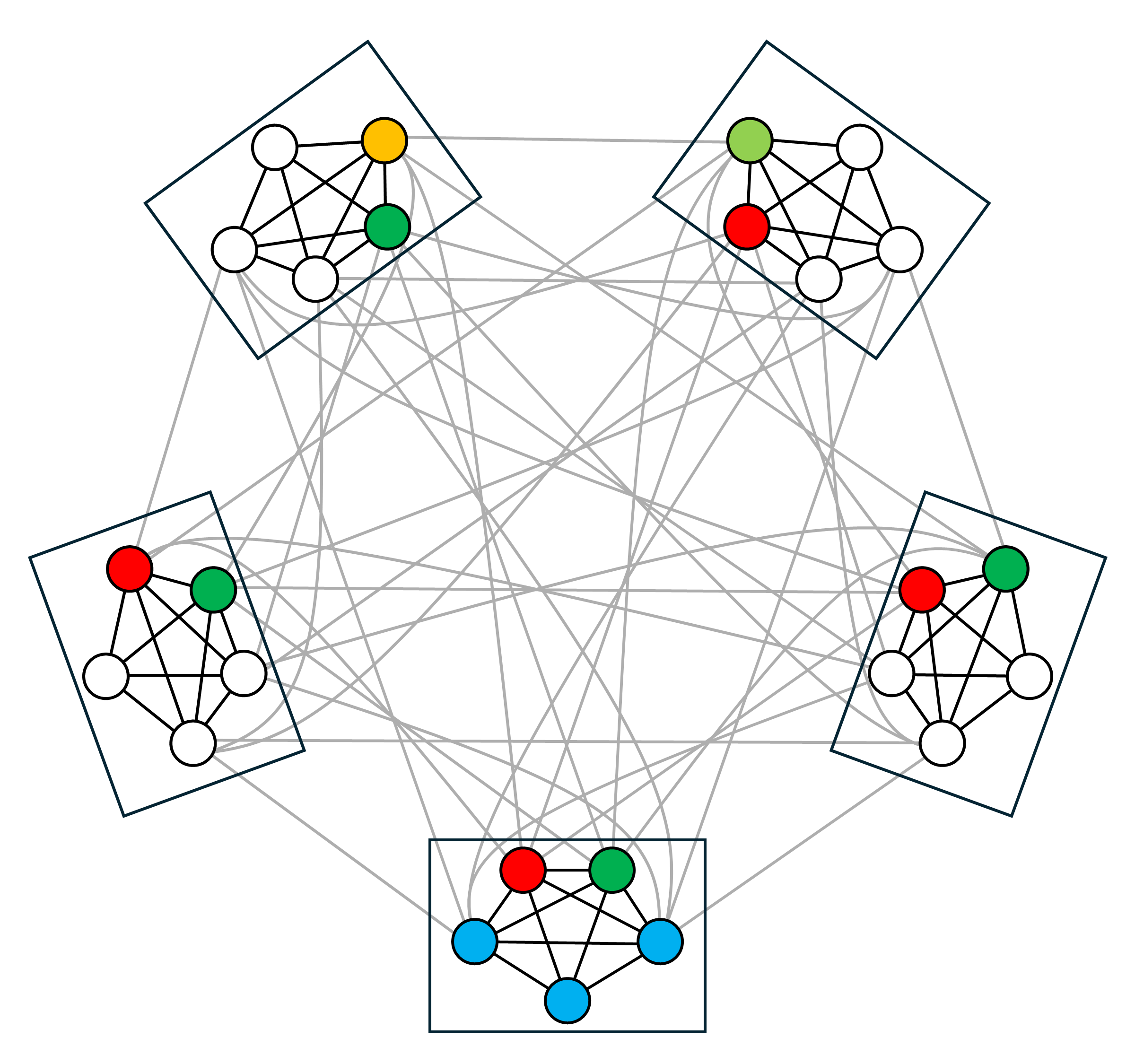}
          \caption{}
          \label{fig:G9_adj}
        \end{subfigure}\hfill
        \begin{subfigure}{0.45\textwidth}
          \includegraphics[width = \textwidth]{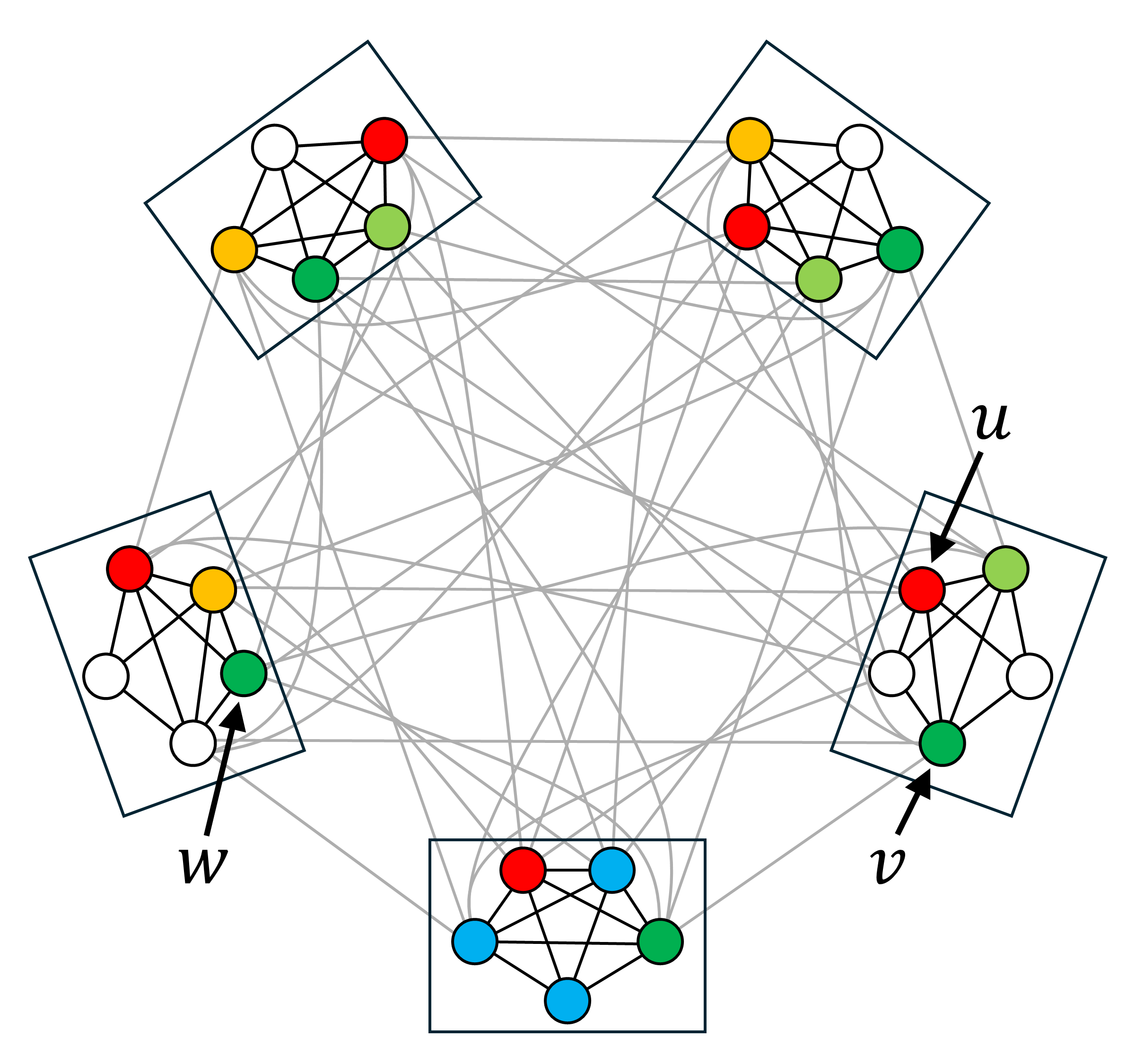}
          \caption{}
          \label{fig:G9_not_adj}
        \end{subfigure}
        \caption{Two cases where the bottom clique has three blue, one red, and one green vertex, depending on if the red and green vertices are adjacent.}
        \label{fig:g9_colorings}
    \end{figure}

    Finally, consider the case where there are two blue, one red, and one green vertex, shown in \cref{fig:g9_colorings}. In \cref{fig:G9_adj}, the red and green vertices are next to each other. For both of these vertices, all four non-clique neighbors must have the same color to ensure the coloring is locally-optimal, and these neighbors are colored in the figure. In particular, the light green vertex at the top must be green, but it is adjacent to four red vertices, so the coloring is not locally-optimal. Similarly, the orange vertex is adjacent to four green vertices.

    In \cref{fig:G9_not_adj}, the red and green vertices are not adjacent. Again, each of the two vertices must share a color with their four non-clique neighbors, which have been colored dark red and green, respectively. Consider vertices $u$ and $v$; at least one of them must be the only vertex of that color in the clique. Suppose without loss of generality that $u$ is the only red vertex in its clique. Then all four of $u$'s non-clique neighbors must be the same color, indicated by the three orange vertices. But now examine vertex $w$, which is in a clique with two red vertices so its non-clique neighbors must all share the same color. This gives us the light green vertices. Notice that the two cliques at the top have two red and two green vertices, so there is no choice for the exterior vertices that makes the coloring locally-optimal.

    Therefore, the only possible locally-optimal colorings assign all five vertices in a clique to the same color.
\end{proof}

\begin{lemma}
    The SG polynomial for \cref{fig:G7} is $k^5 - 5 k^4 + 10 k^3 - 10 k^2 + 4 k$.
\end{lemma}

\begin{proof}
    By Lemma \cref{thm:consistent_color}, all five cliques must have consistent color choices. However, the interior vertices have as many out-clique edges and in-clique edges, and these out-clique edges lead to two adjacent cliques. If two red cliques are adjacent and the next clique over is blue, the blue interior vertices will not be locally-optimal. Therefore, every strict gridlocked coloring of \cref{fig:G7} resembles a proper coloring of a five-cycle. The SG polynomial of $\cref{fig:G7}$ is the chromatic polynomial of the five-cycle; it is well-known that this polynomial is $k(k-1)^4 - k(k-1)^3 + k(k-1)(k-2)$, which can also be written as $k^5 - 5 k^4 + 10 k^3 - 10 k^2 + 4 k$.
\end{proof}

\begin{lemma}
    The SG polynomial for \cref{fig:G9} is $k^5 - 5 k^2 + 4 k$.
\end{lemma}

\begin{proof}
    Each interior vertex is connected to all four of the other cliques. Any assignment of colors to the five cliques will be locally optimal unless four cliques have the same color which is different from the fifth clique, and there are $5k(k-1)$ such colorings. Since we get the SG polynomial from the LO polynomial by subtracting $k$, the SG polynomial for \cref{fig:G9} is $k^5 - 5k(k-1) - k$, which is equivalently $k^5 - 5 k^2 + 4 k$.
\end{proof}

\end{appendices}

\end{document}